\documentclass[12pt]{amsart}
\usepackage {amssymb,latexsym}
\usepackage [all]{xy}

\newtheorem{theorem}{Theorem}%[section]
\newtheorem{lemma}[theorem]{Lemma}

\theoremstyle{definition}
\newtheorem{definition}[theorem]{Definition}

\newtheorem{remark}[theorem]{Remark}

\numberwithin {equation}{section}

\def\cs{{$C^{\ast}$}}

\def\c{{\mathbb{C}}}

\def\gh{{(G,H)}}
\def\ghr{{(G_r,H_r)}}
\def\ghpr{{(G',H')}}

\def\gb{{\overline{G}}}
\def\hb{{\overline{H}}}
\def\ghb{{(\gb,\hb)}}

\def\ucb{{\mathcal{UCB}}}

\def\per{{\text{Per}}}

\def\th{{\theta}}
\def\ff{{\varphi}}

\def\D{{\Delta}}

\def\ba{{\backslash}}
\def\inv{{^{-1}}}

\def\hg{{H\backslash G}}

\begin{document}
\title{Amenability of Hecke pairs and a question of Eymard}

\author{Vahid Shirbisheh}
%\address{}
\email{shirbisheh@gmail.com}

%\subjclass[2000]{(2010) Primary 22D25; Secondary 46L87, 19K99}

%\date{\today}
\keywords{Hecke pairs,  the Schlichting completion, amenability.}

\begin{abstract}
In this short paper we show that if $\gh$ is an amenable Hecke pair and $\Gamma$ is a subgroup of $G$ containing $H$, then the Hecke pair $(\Gamma, H)$ is amenable too. This answers positively a question of Pierre Eymard when it is restricted to Hecke pairs.
\end{abstract}
\maketitle

%\tableofcontents
%%%%%%%%%%%%%%%%%%%%%%%%%%%%%%%%%%      Section        %%%%%%%%%%%%%%%%%%%%%%%%%%%%%%%%%%%%%%%%%%%%%%%%%%%%%%%%%
Let $H$ be a subgroup of a discrete group $G$. The pair $\gh$ is called a Hecke pair if every double coset of $H$ in $G$ is a union of finitely many left cosets of $H$. Some authors refer to this situation by saying that ``$H$ is a Hecke subgroup of $G$'', see \cite{klq}, ``$H$ is an almost normal subgroup of $G$'', see \cite{bc}, or ``$H$ is a commensurated subgroup of $G$'', see \cite{shalomwillis}. We consider a slightly more general setting by allowing $G$ to be a locally compact group and assuming that $H$ is an open subgroup of $G$ which is also a Hecke subgroup of $G$. In this case the homogeneous space $\hg$ is discrete, and so we call the pair $\gh$ a ``discrete'' Hecke pair, as opposed to ``non-discrete'' Hecke pairs which are defined and studied in \cite{s3}. To be consistent with our other works related to Hecke pairs in \cite{s1, s2, s3}, we consider the set $\hg$ of right cosets as the homogeneous space associated with the pair $\gh$, except when $H$ is normal in $G$.

A totally disconnected locally compact group $\gb$ and a compact open subgroup $\hb$ of $\gb$ are associated to every suitable (reduced) discrete Hecke pair $\gh$ such that $G$ is isomorphic with a dense subgroup of $\gb$ and $\hb$ is the closure of the image of $H$ in $\gb$. Then the pair $\ghb$ is called the Schlichting completion of $\gh$. The Schlichting completion of discrete Hecke pairs is a powerful apparatus to reduce the study of an arbitrary discrete Hecke pair $\gh$ to the study of a discrete Hecke pair $\ghb$ in which the subgroup $\hb$ is a compact open subgroup of a totally disconnected locally compact group $\gb$. Many applications of this technique in amenability, Hecke \cs-algebras and property (RD) for Hecke pairs were studied in \cite{tzanev, klq, anan, s3}. In this paper we use the Schlichting completion to study a specific question concerning amenability of pairs in the special case where all pairs are discrete Hecke pairs.

Assume $H$ is a closed subgroup of a locally compact group $G$. Amenability of homogeneous spaces such as $\hg$ was studied thoroughly in \cite{eymard} by Pierre Eymard. Regarding the purpose of this paper, we use the phrase ``amenability of the pair $\gh$'' in lieu of the phrase ``amenability of the homogeneous space $\hg$''. Some authors also refer to this notion by saying that ``$H$ is co-amenable in $G$'', see for instance \cite{monod-popa}. This amenability is a generalization of the amenability of the quotient group $G/H$ when $H$ is normal in $G$. So by reducing the normality condition of $H$ to the condition that $H$ is a Hecke subgroup of $G$, we expect many statements still hold true. Assume $\Gamma$ is a closed subgroup of $G$  containing $H$. If $H$ is a normal subgroup of $G$ and the pair $\gh$ (or equivalently, the quotient group $G/H$) is amenable, then the pair $(\Gamma, H)$ is amenable as well. Motivated by this observation, in page 55 of \cite{eymard}, P. Eymard asked ``whether is this still true when $H$ is no longer normal in $G$?'' The answer is negative in general, see \cite{bekka, monod-popa, pestov} for the history of this question, the counterexamples for the general case and related topics. However, as our main result, when $\gh$ is a discrete Hecke pair, we give a positive answer to Eymard's question in Theorem \ref{thm:main}. Shortly after submitting the first version of this paper to arXiv, we noticed that the same statement as Theorem \ref{thm:main}, with a different proof, had been appeared in Corollary 3.10 of \cite{anan} before.

This paper is organized as follows: We begin with a brief discussion of the Schlichting completions of reduced discrete Hecke pairs. Afterwards, basic definitions and results concerning amenability of (arbitrary or Hecke) pairs $\gh$ are given. The role of the Schlichting completion in the study of amenable discrete Hecke pairs is briefly recalled. Finally, using the Schlichting completion, we prove our main result.

The Schlichting completion of a Hecke pair was studied by K. Tzanev in \cite{tzanev}, based on the works of G. Schlichting in \cite{sch1, sch2}. However, we follow S. Kaliszeweski, M.B. Landstad and J. Quigg's treatment of the subject in \cite{klq} for definitions, notations and basic properties. Given a discrete Hecke pair $\gh$, we set $K_\gh:=\cap_{x\in G}xHx\inv$, $G_r:=\frac{G}{K_\gh}$ and $H_r:=\frac{H}{K_\gh}$. The Hecke pair $\gh$ is called reduced if $K_\gh$ is the trivial subgroup of $G$, otherwise the Hecke pair $\ghr$ is the reduced Hecke pair associated with $\gh$. The Hecke discrete Hecke pair $\ghr$ enjoys many features of $\gh$, see for instance Remark \ref{rem:amen-pairs}(ii), so one can often replace an arbitrary discrete Hecke pair with its reduced version. Given a reduced discrete Hecke pair $\gh$, the group $G$, as a discrete group, can be embedded by a homomorphism inside the group $\per(G/H)$ of all permutations on the set of left cosets $G/H$. Then the completion of $G$ and $H$ in the permutation topology of $\per(G/H)$ are denoted by $\gb$ and $\hb$, respectively. It is shown that $\gb$ is a totally disconnected locally compact group and $\hb$ is a compact open subgroup of $\gb$.

The topology of $\gb$ can be considered as the group topology generated by the set of all conjugates of $\hb$, as a subbasis of open neighborhoods at the identity. This topology on $\gb$ (or its restriction on $G$) is called the Hecke topology. A Hecke pair $\gh$ is called a Schlichting pair if $H$ is a compact and open subgroup of $G$, when $G$ is equipped with the Hecke topology. One notes that when $H$ is a compact open subgroup of a totally disconnected locally compact group $G$, the Hecke pair $\gh$ is a Schlichting pair if and only if it is reduced. The following theorem and the lemma after that are taken from \cite{klq}:

\begin{theorem}
\label{thm:klq4-8} (\cite{klq}, Theorem 4.8)
Let $\gh$ be a reduced Hecke pair with the Schlichting completion $\ghb$. Assume $(L,K)$ is a Schlichting pair. For every homomorphism $\ff:G\to L$ such that $\ff(G)$ is dense in $L$ and $\ff(H)\subseteq K$, there exists a unique continuous homomorphism $\overline{\ff}:\gb \to L$ which extends $\ff$, in other words, $\overline{\ff} \th=\ff$.

Furthermore, if $H=\ff\inv(K)$, then $\overline{\ff}$ is a topological group isomorphism from $\gb$ onto $L$ and maps $\hb$ onto $K$.
\end{theorem}

\begin{lemma}
\label{lem:bijective-Schlich} (\cite{klq}, Proposition 4.9)
Let $\gh$ be a reduced discrete Hecke pair and let $\ghb$ be its Schlichting completion. Then the following statements hold:
\begin{itemize}
\item[(i)] The mapping $\alpha: \hg\to \hb\ba\gb$ (resp. $\alpha':G/H\to \gb / \hb$), defined by $Hg\mapsto \hb g$ (resp. $gH\mapsto g\hb$) for all $g\in G$ is a $G$-equivariant bijection. 
\item[(ii)] The mapping $\beta: G//H \to \gb// \hb$, defined by $HgH\mapsto \hb g\hb$ for all $g\in G$ is a bijection.
\item[(iii)] The mapping $x H x\inv \mapsto x \overline{H}x\inv$ from $\{x H x\inv ; x\in G\}$ into $\{x \overline{H} x\inv ; x\in \overline{G}\}$ is a bijection.
\item[(iv)] Let $\Gamma$ be a closed subgroup of $G$ containing $H$ and let $\overline{\Gamma}^G$ be its closure in $\gb$. Then the same conclusions as above hold for mappings $\alpha':H\ba \Gamma \to \hb\ba \overline{\Gamma}^G$ and $\beta':\Gamma// H\to \overline{\Gamma}^G //\hb$ which are defined by restricting $\alpha$ and $\beta$, respectively.
\end{itemize}
\end{lemma}

\begin{remark}
\label{rem:diff-top-hp}
An important feature of the Schlichting completion associated to a discrete Hecke pair $\gh$ is that the whole process is algebraic and the original topology of the group $G$ plays no role in the final pair. In other words, when $\gh$ is a discrete Hecke pair with respect to two different locally compact topology on $G$ (and $H$),  we obtain the same reduced Hecke pair $\ghr$, (algebraically), and the same Schlichting completion, (algebraically and topologically).
\end{remark}

We proceed with elements of amenability of pairs of locally compact groups.

\begin{definition}
\label{def:amen-pair}
Let $H$ be a closed subgroup of a locally compact group $G$. The pair $\gh$ is called amenable if it possesses the fixed point property, (shortly denoted by ``(FP)''), that is, if $G$ acts continuously on a compact convex subset $Q$ of a locally convex topological vector space by affine transformations and the restriction of this action to $H$ has a fixed point, then there exists a fixed point for the action of $G$ as well.
\end{definition}
In the following remark we address an equivalent definition of amenability of pairs $\gh$ which is important for our purpose. The interested reader can find more equivalent definitions in \cite{eymard}, in Theorem 2.3 of \cite{bekka}, and in Proposition 3 of \cite{monod-popa}. When $\gh$ is a Hecke pair, several other equivalent definitions for amenability of $\gh$ is given in Proposition 5.1 of \cite{tzanev}. For every complex function $f$ on $\hg$ and every $g\in G$, we define $_gf(Hx):=f(Hxg\inv)$ for all $Hx\in \hg$.
\begin{remark}
\label{rem:equi-def-amen}
Let $H$ be a closed subgroup of a locally compact group $G$. Let $\ucb(\hg)$ denote the vector space of bounded functions on $\hg$ which are uniformly continuous with respect to the right action of $G$, that is for every $f\in \ucb(\hg)$, the mapping $g\mapsto\, _gf$ is a continuous map from $G$ into the Banach space of all complex bounded functions on $\hg$ with the uniform convergence topology. A mean on $\ucb(\hg)$ is a linear functional $m:\ucb(\hg)\to \c$ such that $m(1)=1$, $m(\overline{f})=\overline{m(f)}$ for all $f\in \ucb(\hg)$, and $m(f)\geq 0$ whenever $0\leq f\in \ucb(\hg)$. A mean $m$ on $\ucb(\hg)$ is called $G$-invariant if $m(\, _gf)=m(f)$ for all $f\in \ucb(\hg)$ and $g\in G$. The existence of a $G$-invariant mean on $\ucb(\hg)$ is equivalent to the amenability of the pair $\gh$, see Page 12 of \cite{eymard} for more details.
\end{remark}
In addition to the examples and elementary properties of amenable pairs listed in Section 3 of Expos\'{e} 1 of \cite{eymard}, the following observations are useful for studying amenable Hecke pairs:
\begin{remark}
\label{rem:amen-pairs}
 Let $H$ be a closed subgroup of a locally compact group $G$.
\begin{itemize}
\item [(i)] Let $N\subseteq H$ be a normal subgroup of $G$. Set $G':=G/N$ and $H':=H/N$. Then the pair $\gh$ is amenable if and only if the pair $\ghpr$ is amenable. By the fixed point property, the amenability of the pair $\gh$ immediately implies the amenability of the pair $\ghpr$. Conversely, assume that the pair $\ghpr$ is amenable and $m'$ is a $G'$-invariant mean on $\ucb (H'\ba G')$. Define $\ff: \ucb(\hg) \to \ucb(H'\ba G')$ by
\[
\ff(f)(H' \bar{x}):=f(Hx),\qquad \forall f\in \ucb(\hg),\, H' \bar{x} \in H'\ba G',
\]
    where $\bar x$ is the image of $x$ under the quotient map $G\to G/N$ for all $x\in G$. Then one checks that $\ff$ is a linear isomorphism. Using $\ff$, we define a mean $m$ on $\ucb(\hg)$ by $m(f):=m'(\ff(f))$ for all $f\in \ucb(\hg)$. Regarding Remark \ref{rem:equi-def-amen}, we only need to prove that $m$ is $G$-invariant. For every $g\in G$ and $f\in \ucb(\hg)$, we compute
\[
\ff(\, _gf)(H'\bar{y}) = f(Hyg\inv)=\ff(f)(H'\bar y \bar g \inv)=\, _{\bar g}(\ff(f))(H' \bar y),
\]
    for all $H'\bar{y}\in H' \ba G'$. Hence we have
\[
m(\, _gf)= m'(\ff(\, _gf))= m'(\, _{\bar g}(\ff(f)))=m' (\ff(f))=m(f).
\]
\item [(ii)] It follows immediately from (i) that a discrete Hecke pair $\gh$ is amenable if and only if the reduced Hecke pair $\ghr$ associated to $\gh$ is amenable.
\item [(iii)] It was proved in Proposition 5.1 of \cite{tzanev} that a reduced discrete Hecke pair $\gh$ is amenable if and only if its Schlichting completion $\ghb$ is amenable. Since $\hb$ is compact, the amenability of these Hecke pairs is equivalent to the amenability of the totally disconnected locally compact group $\gb$.
\item [(iv)] Given a discrete Hecke pair $\gh$, let $\ghr$ denote the reduced Hecke pair associated to $\gh$ and let $(\overline{G_r}, \overline{H_r})$ denote the Schlichting completion of the latter Hecke pair. Using the above discussions, we conclude that the Hecke pair $\gh$ is amenable if and only if $\overline{G_r}$ is amenable.
\item [(v)] Using the above item and Remark \ref{rem:diff-top-hp}, we conclude that if $\tau_1$ and $\tau_2$ are two locally compact topologies on a group $\Gamma$ and $\Gamma_0$ is a Hecke subgroup of $\Gamma$ such that $\Gamma_0$ is open in $\Gamma$ with respect to both topologies $\tau_1$ and $\tau_2$, then the discrete Hecke pair $(\Gamma, \Gamma_0)$ is amenable with respect to the topology $\tau_1$ if and only if it is amenable with respect to the topology $\tau_2$.
\item [(vi)] Let $H_0$ be a closed subgroup of a locally compact group $G_0$ and let $\iota:G_0\to G$ be a continuous, dense, and one-to-one homomorphism such that $H=\overline{\iota(H_0)}$. Then using (FP), one easily observes that the amenability of the pair $(G_0,H_0)$ implies the amenability of the pair $\gh$.
\item [(vii)] Let the pair $(G_0,H_0)$ and the mapping $\iota$ be as the above item. If $\iota$ is a homeomorphism from $G_0$ onto its image under $\iota$, then using (FP), it is easily seen that the amenability of the pair $\gh$ implies the amenability of the pair $(G_0,H_0)$.
\end{itemize}
\end{remark}

\begin{theorem}
\label{thm:main} Let $\gh$ be a discrete Hecke pair and let $\Gamma$ be a closed subgroup of $G$ containing $H$. If the pair $\gh$ is amenable, then the pair $(\Gamma,H)$ is amenable too.
\end{theorem}
\begin{proof}
Regarding Remark \ref{rem:amen-pairs}(ii), without loss of generality, we can assume that the Hecke pair $\gh$ is reduced. Since we are going to deal with two Schlichting completions and we need to distinguish between them, we denote the Schlichting completion of $\gh$ by $(\widetilde{G},\widetilde{H})$ and closure of $\Gamma$ in $\widetilde{G}$ by $\widetilde{\Gamma}$. We also set
\begin{eqnarray*}
K&:=&K_{(\Gamma, H)}= \bigcap_{x\in \Gamma} xHx\inv\\
N&:=&K_{(\widetilde{\Gamma}, \widetilde{H})}= \bigcap_{x\in \widetilde{\Gamma}} x\widetilde{H} x\inv= \bigcap_{x\in \Gamma} x\widetilde{H} x\inv.
\end{eqnarray*}
By Remark \ref{rem:amen-pairs}(iii), $\widetilde{G}$ is amenable. It implies that $\widetilde{\Gamma}$ is amenable too. Since $\widetilde{H}$ is compact and open, the discrete Hecke pair $(\widetilde{\Gamma},\widetilde{H})$ is amenable. By Remark \ref{rem:amen-pairs}(i), the reduced discrete Hecke pair $(\frac{\widetilde{\Gamma}}{N},\frac{\widetilde{H}}{N})$ is amenable. If we show that this latter Hecke pair is isomorphic to the Schlichting completion $\left(\overline{(\frac{\Gamma}{K})},\overline{(\frac{H}{K})}\right)$ of $(\frac{\Gamma}{K},\frac{H}{K})$, then it follows from Remark \ref{rem:amen-pairs}(iv) that the Hecke pair $(\Gamma, H)$ is amenable and our proof is complete.

To prove the above isomorphism, we define a map $\ff: \frac{\Gamma}{K} \to \frac{\widetilde{\Gamma}}{N}$ by $\ff(xK):=xN$ for all $x\in \Gamma$. Using Lemma \ref{lem:bijective-Schlich}, it is straightforward to check that $\ff$ is well defined, its image is dense in $\frac{\widetilde{\Gamma}}{N}$, and $\frac{H}{K}=\ff\inv (\frac{\widetilde{H}}{N})$. Therefore, by Theorem \ref{thm:klq4-8}, $\ff$ extends to a topological group isomorphism $\overline{\ff}:\overline{(\frac{\Gamma}{K})}\to  \frac{\widetilde{\Gamma}}{N}$ such that $\overline{\ff}\left(\overline{(\frac{H}{K})}\right)=\frac{\widetilde{H}}{N}$.
\end{proof}
We conclude this paper with discussing how amenability of pairs of groups behaves with respect to commensurability of subgroups. Two subgroups $H$ and $K$ of a group $G$ are called commensurable if $H\cap K$ is a finite index subgroup of both $H$ and $K$. In \cite{s2}, we showed that property (RD) of a Hecke pair $\gh$ is preserved if we replace $H$ by another Hecke subgroup $K$ provided that $H$ and $K$ are commensurable. In the following remark we answer a similar question concerning amenable pairs, but we do not have to restrict ourselves to Hecke pairs.

\begin{remark}
\label{rem:amen-commen}
Let $H$ be a closed cocompact subgroup of a locally compact group $G$. Then the pair $\gh$ is amenable if and only if $\D_G|_H=\D_H$, see Page 17 of \cite{eymard}. In particular if $H$ is a finite index subgroup of $G$, then the pair $\gh$ is amenable.

The relation of commensurability of subgroups preserves the amenability of corresponding pairs. More precisely, let $H$ and $K$ be  two closed commensurable subgroups of a locally compact group $G$. Then the pair $\gh$ is amenable if and only the pair $(G,K)$ is amenable. To prove this, without loss of generality, we can assume that $K$ is a finite index subgroup of $H$. Then using (FP), one easily observes that the amenability of $\gh$ follows from the amenability of $(G,K)$. The converse implication follows from the above discussion. 

With the same proof, one can generalize this statement to the case that $H$ and $K$ are two closed subgroups of $G$ such that $H\cap K$ is a cocompact subgroup of both groups $H$ and $K$, and we have $\D_H|_{H\cap K}=\D_{H\cap K}=\D_K|_{H\cap K}$.
\end{remark}
%%%%%%%%%%%%%%%%%%%%%%%%%%%%%%%%%%          bibliography          %%%%%%%%%%%%%%%%%%%%%%%%%%%%%%%%%%%%%%%%%%%%%%%%

\bibliographystyle{amsalpha}

\end{document}